\documentclass[12pt]{amsart} %

\usepackage{float}

\usepackage{epsfig}
\usepackage{tikz}
\usepackage{array,amsmath, enumerate,  url, psfrag}
\usepackage{amssymb, amsaddr, fullpage}
\usepackage{graphicx,subfigure}
\usepackage{color}
\restylefloat{figure}

\newtheorem{theorem}{Theorem}[section]

\newtheorem{corollary}[theorem]{Corollary}

\newtheorem{thm}{Theorem}[section]
\newtheorem{lem}[thm]{Lemma}

\newtheorem{definition}[thm]{Definition}
\newtheorem{example}[thm]{Example}

\title{Every toroidal graph without $3$-cycles is odd 7-colorable}

\author{Fangyu Tian$^{1}$\hskip 0.2in  Yuxue Yin$^{1}$}

\address{
$^{1}$\small Department of Mathematics, Central China Normal University, Wuhan, Hubei, China.\\
$^2$\small Department of EE, Tsinghua University, Beijing, China.
}
\thanks{The research of the last author was supported in part by the Natural Science Foundation of China (11728102) and the NSA grant H98230-16-1-0316.}

\email{yinyuxue945@mail.tsinghua.edu.cn}

\begin{document}

\maketitle
\section{Abstract}
Odd coloring is a proper coloring with an additional restriction that every non-isolated vertex has some color that appears an odd number of times in its neighborhood. The minimum number of colors $k$ that can ensure an odd coloring of a graph $G$ is denoted by $\chi_o(G)$. We say $G$ is $k$-colorable if $\chi_o(G)\le k$. This notion is introduced very recently by Petruševski and Škrekovski, who proved that if $G$ is planar then $ \chi_{o}(G) \leq 9 $. 
A toroidal graph is a graph that can be embedded on a torus. Note that a $K_7$ is a toroidal graph, $\chi_{o}(G)\leq7$. 
In this paper, we proved that, every toroidal graph without $3$-cycles is odd $7$-colorable. Thus, every planar graph without $3$-cycles is odd $7$-colorable holds as a corollary. That's to say, every toroidal graph is $7$-colorable can be proved if the remained cases around $3$-cycle is resolved.  
\section{Introduction}
In this paper, all graphs are finite and simple, which means no parallel edges and no loops at their vertices. 
An odd coloring of a graph is a proper coloring with the additional constraint that each non-isolated vertex has at least one color that appears an odd number of times in its neighborhood. We use $c_o(v)$ to denote the color and $C_o(v)$ to denote the set of the odd colors. A graph G is odd $c$-colorable if it has an odd $c$-coloring. The odd chromatic number of a graph G, denoted by $\chi_o(G)$, is the minimum $c$ such that G has an odd $c$-coloring. 

Odd coloring was introduced very recently by Petru$\breve{s}$evski and $\breve{S}$krekovski~\cite{petruvsevski2021colorings}, who proved that planar graphs are odd $9$-colorable and conjecture every planar graph is $5$-colorable. Petr and Portier~\cite{petr2022odd} proved that planar graphs are odd $8$-colorable. Fabrici\cite{fabrici2022proper} proved a strengthening version about planar graphs regarding similar coloring parameters.  
A toroidal graph is a graph that can be embedded on a torus. Note that a $K_7$ is a toroidal graph, $\chi_{o}(G)\geq7$.


In this paper, we prove that every toroidal graph without $3$-cycles is $7$-colorable, as is shown in Theorem~\ref{th1}. If the remained cases around $3$-cycles is resolved, then every toroidal graph is odd $7$-colorable can be proved. 
\begin{theorem}\label{th1}
 If $G$ is a toroidal graph without  $3$-cycles, then $\chi_o(G)\leq 7$. 
\end{theorem}
Note that planar graph is a subclass of toroidal graph, the corollary holds anyway.
\begin{corollary}
 If $G$ is a planar graph without  $3$-cycles, then $\chi_o(G)\leq7$.
\end{corollary}

Let $V(G),E(G)$ and $F(G)$ denoted the sets of vertices, edges and faces of $G$, respectively. A {\em  $k$-vertex} ({\em  $k^{+}$-vertex} or {\em  $k^{-}$-vertex}) is a vertex of degree $k$ (at least $k$ or at most $k$). Similarly, a {\em  $k$-face} ({\em  $k^{+}$-face} or {\em  $k^{-}$-face}) is a face of degree $k$ (at least $k$ or at most $k$). A {\em  $k$-neighbor} ({\em  $k^{+}$-neighbor} or {\em  $k^{-}$-neighbor}) of $v$ is a $k$-vertex ({\em $k^{+}$-vertex} or {\em  $k^{-}$-vertex}) adjacent to $v$.  For $f\in F(G)$, $f=[v_1v_2\ldots v_k]$ denotes that $v_1,v_2,\ldots,v_k$ are the vertices lying on the boundary of $f$ in clockwise  order.  A $k$-face $[v_1v_2\ldots v_k]$ is called an $(m_1,m_2,\ldots,m_k)$-face if $d(v_i)=m_i$ for $i\in\{1,2,\ldots,k\}$.

For a $4^+$-vertex $v$, we say $v$ is {\em convenient} if $v$ is an odd vertex or an even vertex with  a $2$-neighbor, otherwise $v$ is  {\em non-convenient}.
 A {\em  $k_i$-vertex} is a  $k$-vertex adjacent to $i$ $2$-vertices. A {\em  $k_i$-face} is a  $k$-face incident with $i$ $2$-vertices.

\section{Reduce configure}
Let $G$ be a counterexample to Theorem~\ref{th1} with the minimum number of $4^+$-vertices, and subject to that, the number of $4^+$-neighbors of $4^+$-vertices $G$ is minimized, and subject to these conditions $|E(G)|$ is minimized.
\begin{lem}\label{no-3-v}
 $G$ has no $3$-vertices.
\end{lem}
\begin{proof}
  Suppose that $G$ has a $3$-vertex $v$. We assume that $v_1,v_2$ and $v_3$ are the neighbors of $v$. Let $G'=G-v$. Since the number of $4^+$-vertices and $4^+$-neighbors of $4^+$-vertices in $G'$ does not increase and $G'$ has fewer edges than $G$, $G'$ has an odd $7$-coloring $c'$ by the minimality.   Color each vertex other than $v$ in $G$  with the same color in $G'$ and color  $v$ with $\lbrack 7 \rbrack\setminus\{c'(v_1),c'(v_2),c'(v_3),c'_o(v_1),c'_o(v_2),c'_o(v_3)\}$. Since $v$ is a $3$-vertex, $v$ must have an odd coloring.  Then the coloring of $G'$ can return back to $G$, a contradiction.
\end{proof}

\begin{lem} \label{conven}
If $v$ is convenient, then $v$ always admits an odd coloring.
\end{lem}
\begin{proof}
By the definition of convenient vertex, $v$ is either an odd vertex or an even vertex
with a $2$-neighbor. If $v$ is an odd vertex, then $v$ always admits an odd coloring. If $v$ is an even vertex with a $2$-neighbor $v'$, then $v$ always admits an odd coloring via recoloring $v'$ with
a color in $[7]\setminus  \{c_o(v), c(v),  c(u), c_o(u)\}$, where $u$ is the other neighbor of $v'$ and $c_o(v)\neq c_o(v')$. If $c_o(v)= c_o(v')$, keep the color of $v$.
 \end{proof}


\begin{lem}\label{tool} (Tool Lemma 2) The following statements hold:
\begin{enumerate}[(1)]
   \item Any two $2$-vertices are not adjacent.
    \item Any two convenient vertices are not adjacent.
    \item A $k$-vertex is adjacent to at most  $2k-7$ $2$-vertices or convenient vertices, where $4\leq k \leq 6$.
   \end{enumerate}
\end{lem}
\begin{proof}
\begin{enumerate}[(1)]
    \item Suppose otherwise that two $2$-vertices $u$ and $v$ are adjacent. Let $z'$ be the other neighbor of $z$ for $z\in\{u,v\}$. Let $G'=G-v$. Then $G'$ has an odd $7$-coloring $c'$ by the minimality of $G$.   Color each vertex other than $v$ in $G$ with the same color in $G'$ and color $v$ with $[7]\setminus\{c'(u),c'(v'),c_o'(v'),c'(u')\}$.
    Then the odd $7$-coloring of $G'$ can return back to $G$,  a contradiction.

\item Suppose otherwise that two convenient vertices $u$ and $v$ are adjacent. By the definition of convenient vertices, $u$ and $v$ are $4^+$-vertices. Let $G'$ be the graph obtained from $G$ by
splitting edge $uv$ with a $2$-vertex $w$. Since the number of $4^+$-vertices in $G'$ does not increase and $4^+$-vertices $u$ and $v$ have fewer $4^+$-neighbors in
$G'$ than $G$, there is an odd $7$-coloring $c'$ of $G'$ by the minimality of $G$. Note that $c'(u)  \neq c'(v)$
since $w$ is a $2$-vertex. Let $c(z) = c'(z)$ for $z \in V(G)$. Then $c$ is an odd coloring of $G$ since
$u$ and $v$ has an odd coloring by Lemma \ref{conven}.
   \item Suppose otherwise that a $k$-vertex $v$ is adjacent to at least $2k-6$ $2$-vertices or   convenient vertices  where $4\leq k \leq 6$. 
       Let $G'=G-v$. Then $G'$ has an odd $7$-coloring $c$ by the minimality of $G$.
 We want to color $v$ such that $v$ admits an odd $7$-coloring while each other $4^+$-vertex in $G$ has the same color as in $G'$ and some $2$-vertices can be recolored to ensure its neighbors' proper and odd coloring.
Note that if $v$ has a $2$-neighbor $v'$, then we only need to avoid the color of the other neighbor of $v'$ in  coloring $v$. After that, recolor $v'$ with the color  different from $v$ and the other neighbor of $v'$.
   If $v$ has a convenient neighbor $v'$, then we only need to avoid the color of   $v'$ in  coloring $v$ since $v'$ always has odd coloring by Lemma \ref{conven}.    If $v$ has a neighbor $v'$ which is neither $2$-neighbor nor convenient neighbor, then we need to avoid two colors ($c(v')$ and $c_o(v')$)   in  coloring $v'$.
   Thus, $v$ has at most $(2k-6)+2(k-(2k-6))=6$ forbidden colors,
   and there is one color left for $v$. Then the coloring of $G'$ can return back to $G$, a contradiction.
   \end{enumerate}
  \end{proof}
  \begin{lem}\label{4-f}
   If $f$ is a $4$- or $5$-face, then $f$ is incident with at most one $2$-vertex.
  \end{lem}
\begin{proof}
First suppose that $4$-face $f=[v_1v_2v_3v_4]$ is incident with two $2$-vertices. By Lemma \ref{tool}(1), we assume that $v_1$ and $v_3$ are $2$-vertices. Let $G'=G-v_1$. Then $G'$ has an odd $7$-coloring $c$.   Color each vertex other than $v_1$ in $G$  with the same color in $G'$ and color  $v$ with $\lbrack 7 \rbrack\setminus\{c(v_2),c(v_4),c_o(v_2),c_o(v_4)\}$. Since $v_3$ is a $2$-vertex and $v_3$ has at least one odd coloring, $c(v_2)\neq c(v_4)$. Then $v_1$ also has an odd coloring.  Then the coloring of $G'$ can return back to $G$, a contradiction.

  Next suppose  that $5$-face $f = [v_1v_2v_3v_4v_5]$ is incident with two $2$-vertices. By Lemma \ref{tool}(1), we
assume that $v_1$ and $v_3$ are $2$-vertices.  Then $v_5$ and $v_4$ are convenient,  which contradicts Lemma \ref{tool}(2).
\end{proof}
  
\begin{lem}\label{n-c-4-v}
  If  $v$ is  a non-convenient $4$-vertex and $v$ is adjacent to exactly one convenient vertex, then  $v$ is adjacent to three non-convenient $6^+$-vertices.
\end{lem}
\begin{proof}
By Lemma~\ref{tool}, a $4$-vertex is incident with at most one convenient vertex. Thus if $v$ is adjacent to exactly one convenient vertex, then the other three vertices are non-convenient vertices. Suppose that $v$ is adjacent to one non-convenient $5^-$-vertex. Note that $5$-vertex is convenient by definition, this non-convenient vertex must be a $4$-vertex. Let $v_1,v_2,v_3$ and $v_4$ be the neighbors of $v$, $v_1$ be a convenient vertex,  $v_2,v_3,v_4$ be non-convenient  vertices and $v_2$ be  a $4$-vertex.  Let $v_2',v_2'',v_2'''$ be other neighbors of $v_2$. Let $G'$ be the graph obtained from $G-\{v,v_2\}$ by adding $2$-paths $v_1x_1v_3,v_1x_2v_4,v_3x_3v_4,v_2'y_1v_2'',v_2''y_2v_2''',v_2'''y_3v_2'$ where $x_i$ and $y_j$ are $2$-vertices for $i,j=1,2,3$. Since $G'$ has fewer $4$-vertices than $G$, $G'$ has an odd $7$-coloring $c$ by the minimality of $G$. Since $x_i$ and $y_j$ are $2$-vertices for $j=1,2,3$, $c(v_1)\neq c(v_3)\neq c(v_4)$ and  $c(v_2')\neq c(v_2'')\neq c(v_2''')$. Color $v_2$ with color in $[7]\setminus \{c(v_2'),c_o(v_2'),c(v_2''),c_o(v_2''),c(v_2'''),c_o(v_2''')\}$. Then $v_2$ is proper and always have an odd color regardless the color of $v$. Then color $v$ with color in $[7]\setminus \{c(v_1),c(v_2),c(v_3),c_o(v_3),c(v_4),c_o(v_4)\}$.  Then $v$ is proper and always have an odd color since $c(v_1)\neq c(v_3)\neq c(v_4)$. Then $G$ has an odd $7$-coloring, a contradiction.
\end{proof}

 \begin{lem}\label{6-v}
  If  $v$ is  a non-convenient $k$-vertex and $v$ is adjacent to at least $\frac{3k}{2}-5$ convenient vertices, then  none of  convenient neighbors of $v$ is a $5_3$-vertex for $k=4,6,8,10$.
\end{lem}

\begin{proof}
Suppose otherwise  one of  convenient neighbors of $v$ is a $5_3$-vertex. Let $v_1,v_2,\ldots,v_k$ be the    neighbors of $v$ and  $v_1,v_2,\ldots,v_{\frac{3k}{2}-5}$ be convenient and $v_1$ be a $5_3$-vertex. Let $G'$ be the graph obtained from $G$ by
splitting edge $vv_1$ with a $2$-vertex $z$. Since   the number of $4^+$-vertices in $G'$ does not increase and $4^+$-vertices $v$ and $v_1$ have fewer $4^+$-neighbors in
$G'$, there is an odd $7$-coloring $c$ of $G'$ by the minimality of $G$. Color each vertex in $G$ with the same color as in $G'$.  Since $v_1$ is convenient, $v_1$ has an odd coloring by Lemma \ref{conven}. If $v$ has an odd coloring, then $G$ is odd $7$-colorable, a contradiction. Thus, $v$ has no odd coloring.  Since   $z$ has an odd coloring in $G'$, $c(v)\neq c(v_1)$. We assume $c(v)=1,c(v_1)=2$. We may assume that $\{c(v_1),c(v_2),\ldots,c(v_k)\}=\{2,3,\ldots,\frac{k}{2}+1\}$ since $v$ has no odd coloring. Let 
$v_1'$ be the $4^+$-neighbor of $v_1$ and $v_1'',v_1''',v_1''''$ be  $2$-neighbors of $v_1$. Then Let 
$\{c_1,c_2,c_3,c_4,c_5\}=\{c(v_1'),c_o(v_1'),c_o(v_1''),c_o(v_1'''),c_o(v_1'''')\}$. If $3,4,5,6$ or $7$ is not in the set $\{c_1,c_2,c_3,c_4,c_5\}$, then recolor $v_1$ with color $3,4,5,6$ or $7$ respectively. Then $v$ has an odd coloring, a contradiction. Thus, $\{c_1,c_2,c_3,c_4,c_5\}=\{3,4,5,6,7\}$. Then recolor $v$ with color in $[7]\setminus\{1,2,3,\ldots,\frac{k}{2}+1,c_o(v_{\frac{3k}{2}-5+1}),\ldots,c_o(v_k)\}$, $v_1$ with $1$. Note that if $k=4$, then $\frac{3k}{2}-5=1$, then $\{1,2,3,\ldots,\frac{k}{2}+1,c_o(v_{\frac{3k}{2}-5+1}),\ldots,c_o(v_k)\}=\{1,2,3,
c_o(v_2),\\ c_o(v_3),c_o(v_4)\}$; if $k=6$, then $\frac{3k}{2}-5=4$, then $\{1,2,3,\ldots,\frac{k}{2}+1,c_o(v_{\frac{3k}{2}-5+1}),\ldots,\\ c_o(v_k)\}=\{1,2,3,4,
c_o(v_5),c_o(v_6)\}$; if $k=8$, then $\frac{3k}{2}-5=7$, then $\{1,2,3,\ldots,\\ \frac{k}{2}+1,c_o(v_{\frac{3k}{2}-5+1}),\ldots,c_o(v_k)\}=\{1,2,3,4,5,
c_o(v_8)\}$; if $k=10$, then $\frac{3k}{2}-5=10$, then $\{1,2,3,\ldots,\frac{k}{2}+1,c_o(v_{\frac{3k}{2}-5+1}),\ldots,c_o(v_k)\}=\{1,2,3,4,5,
6\}$. In each case, $v$ has at least one color. Then $v$ has an odd color $1$, a contradiction.
\end{proof}
\begin{lem}\label{5-f}
 Let  $uwv$ be a $2$-path, $w$ be a $2$-vertex,  $f_1$ and $f_2$ be two faces incident with $uwv$. If $u$ and $v$  are   $5_3$-vertices,   then at most one of $f_1$ and $f_2$ is a $4$-face.  
  \end{lem}
\begin{proof}
  Suppose otherwise that $f_1$ and $f_2$ are   $4$-faces. Let $x$ and $y$ be the left vertex of $f_1$ and $f_2$ respectively. Since $u$ and $v$ are convenient, $x$ and $y$ are non-convenient $6^+$-vertices by Lemma \ref{tool}(2)(3). Let $G'$ be the graph obtained from $G$ by
splitting edge $xv$ with a $2$-vertex $z$. Since the number of  $4^+$-vertices does not increase and $4^+$-vertices $x$ and $v$ have fewer $4^+$-neighbors in
$G'$, there is an odd $7$-coloring $c$ of $G'$ by the minimality of $G$. Color each vertex in $G$ with the same color as in $G'$.  Since $v$ is convenient, $v$ has an odd coloring by Lemma \ref{conven}. If $x$ has an odd coloring, then $G$ is odd $7$-colorable, a contradiction. Thus, $x$ has no odd coloring.  Since $z$ and $w$ have an odd coloring in $G'$, $c(x)\neq c(v)\neq c(u)$. We assume $c(x)=1,c(u)=2,c(v)=3$
Let $u',u''$ be other $2$-neighbors of $u$ and   $v',v''$ be other $2$-neighbors of $v$,  $\{c_1,c_2\}=\{c_o(u'),c_o(u'')\}$ and $\{c_3,c_4\}=\{c_o(v'),c_o(v'')\}$.

 First assume that $c(y)=c(x)$. Then recolor $u$ with the color in $[7]\setminus\{1,2,3, c_o(y),c_1,c_2\}$, $w$ with different color from the current color of $u$ and $v$. Then $G$ has an odd $7$-coloring, a contradiction.

 Next assume that   $c(y)\neq c(x)$. We assume that $c(y)=4$. If $c_1\in\{1,2,3,4\}$, then recolor $u$ with the color in $[7]\setminus\{1,2,3,4,c_2,c_o(y)\}$, $w$ with different color with current color of $u$ and $v$. Then $G$ has an odd $7$-coloring, a contradiction. Thus, $c_1\notin\{1,2,3,4\}$. By symmetry,  $c_2\notin\{1,2,3,4\}$. We assume that $c_1=5,c_2=6$. Then recolor $u$ with the color $7$, and recolor $w$ with a different color from the current color of $u$ and $v$. In this case, if $y$ has an odd coloring, then $G$ has an odd $7$-coloring, a contradiction. Thus, $y$ has no odd color in the case of $u$ with color $7$ and $v$ with color $3$. Then $y$ has an odd number of neighbors ( other than  $u$ and $v$) with color $3$ and $7$. Then recolor $v$ with a color in $[7]\setminus\{1,2,3,4,c_3,c_4,\}$ and keep color of $u$  with $2$, recolor $w$ with a different color from the current color of $u$ and $v$. Then $y$ has an odd color $3$ and $x$ has an odd color $c(v)$. a contradiction. 
 \end{proof}

\begin{lem}\label{4_0-face}
   There is no $(4_0,4_0,4_0,4_0)$-face $f$ in $G$. 
\end{lem}
\begin{proof}
  Suppose otherwise that $G$ has a $(4_0,4_0,4_0,4_0)$-face $f$. Let $f=[v_1v_2v_3v_4]$ and $v_i',v_i''$ be the other neighbors of $v_i$ for $i=1,2,3,4$. Let $G'$ be the graph obtained from $G-\{v_1,v_2\}$  by adding $2$-paths $v_1'x_1v_1'',v_1'x_2v_4,v_1''x_3v_4,v_2'y_1v_2'',v_2'y_2v_3$ and $v_2''y_3v_3$ where $x_j$ and $y_j$ are $2$-vertices for $i,j=1,2,3$.  
  Since the distance of vertices $v_1',v_1',v_4$ and  $v_2',v_2',v_3$ does not decrease in $G'$, $G'$ has no $3$-cycles.  Since $G'$ has fewer $4^+$-vertices than $G$, $G'$ has an odd $7$-coloring $c$ by the minimality. Since $x_j$ and $y_j$ are $2$-vertices for $j=1,2,3$, $c(v_1')\neq c(v_1'')\neq c(v_4)$ and $c(v_2')\neq c(v_2'')\neq c(v_3)$. Color $v_2$ with color in $[7]\setminus \{c(v_2'),c_o(v_2'),c(v_2''),c_o(v_2''), c(v_3),c_o(v_3)\}$. Then $v_2$ must have an odd color regardless the color of $v_1$ since $c(v_2')\neq c(v_2'')\neq c(v_3)$.  If $c(v_2)\in \{c(v_1'),c_o(v_1'),c(v_1''),c_o(v_1''), c(v_4)\}$, then color $v_1$ with color in $[7]\setminus \{c(v_1'),c_o(v_1'),c(v_1''),c_o(v_1''),\\ c(v_4),c_o(v_4)\}$. Then $G$ has an odd $7$-coloring, a contradiction. Thus,  $c(v_2)\notin \{c(v_1'),c_o(v_1'),\\c(v_1''),c_o(v_1''), c(v_4)\}$. Then    color $v_1$  with color in $[7]\setminus \{c(v_1'),c_o(v_1'),c(v_1''),c_o(v_1''), c(v_4),c(v_2)\}$. In this case, $v_1$ must have an odd color since $c(v_1')\neq c(v_1'')\neq c(v_4)$.  If $v_4$ has an odd color, then $G$ has an odd $7$-coloring, a contradiction. Thus, $v_4$ has no odd color. Then the all neighbors of $v_4$ occupy two colors. We assume  these two colors are $c_1$ and $c_2$. Then we recolor $v_1$ with the current color of $v_4$, recolor $v_4$ with color in $[7]\setminus \{c_1,c_2,c_o(v_4'),c_o(v_4''), c_o(v_3),c(v_1)\}$, where $c(v_1)$ is the current color of $v_1$.  Since $c(v_1')\neq c(v_1'')\neq c(v_2)$, $v_1$ has an odd color. Since we recolor $v_1$, $v_4$ has an  odd color. Then $G$ has an odd $7$-coloring, a contradiction.
\end{proof}


 \section{Proof of Theorem~\ref{th1}}
  We are now ready to  complete the proof of Theorem~\ref{th1}.   Let each   $v\in V(G)$ has an initial charge of $\mu(v)=2d(v)-6$, each $f\in\cup F(G)$ have an initial charge of $\mu(f)=d(f)-6$.   By Euler's Formula, $|V(G)|+|F(G)|-|E(G)|\geq 0$. Then $\sum_{v\in V }\mu(v)+\sum_{f\in F }\mu(f)=0$.

Let $\mu^*(x)$ be the charge of $x\in V(G)\cup F(G)$ after the discharge procedure. To lead to a contradiction, we shall prove that $\sum_{x\in V(G)\cup F(G)}  \mu^*(x) > 0$.  Since the total sum of charges is unchanged in the discharge procedure, this contradiction proves Theorem~\ref{th1}.


  A $2$-vertex is {\em bad} if it is incident with two $4$-faces, {\em semi-bad} if it is incident with exactly one $4$-face, {\em non-bad} otherwise. A $5_3$-vertex $v$ is {\em poor} to $u$ if $u$ is non-convenient and two faces incident with $uv$ are $4_1$-faces.

  {\bf {Discharging Rules}}
  \begin{enumerate}[(R1)]
    \item Every vertex sends $\frac{1}{2}$ to each incident $4$-face; every non-convenient vertex sends $\frac{1}{2}$ to each incident $5$-face.
    \item Every $4^+$-vertex sends $\frac{3}{2}$ to each adjacent bad $2$-vertex, $\frac{5}{4}$ to each adjacent semi-bad $2$-vertex, $1$ to each adjacent non-bad $2$-vertex.
    \item Every non-convenient vertex  $v$ sends $1$ to each adjacent poor $5_3$-vertex,  $\frac{1}{2}$ to each other adjacent convenient vertex.

  \item Every non-convenient $6^+$-vertex sends $\frac{1}{4}$ to each adjacent  non-convenient $4$-vertex  which is adjacent to one non-convenient vertex.
  \end{enumerate}
  Next, we check the final charge of each vertex and each face has non-negative charge.
  Let $v_1,v_2,\ldots,v_k$ be  neighbors of $k$-vertex $v$, $f_i$ be the face incident with $v_i,v$ and $v_{i+1}$ and  $u_i$ be the other neighbor if $v_i$ is a $2$-vertex  for $i=\{1,2,\ldots,k\}$, especially force $i+1=1$ if $i=k$.
\begin{enumerate}[1.]
  \item Let $f$ be a $4$-face. By (R1), each vertex sends $\frac{1}{2}$ to $f$. Thus, $\mu^*(f)=4-6+\frac{1}{2}\times4=0$.
  \item Let $f$ be a $5$-face. By Lemma \ref{tool}(1)(2), $f$ is incident with at least two non-convenient vertex. By (R1), each non-convenient vertex sends $\frac{1}{2}$ to $f$. Thus, $\mu^*(f)\geq 5-6+\frac{1}{2}\times2=0$.
  \item Let $f$ be a $6^+$-face. By Rules, $f$ is not involved any discharging. Thus, $\mu^*(f)\geq d(f)-6\geq0$.
  \item Let $v$ be a $2$-vertex. By Lemma \ref{tool}(1), $v$ is not adjacent to any $2$-vertex. By Lemma \ref{no-3-v}, each neighbor of $v$ is a $4^+$-vertex. If $v$ is bad, then $v$ is incident with two $4$-faces. By (R1) and (R2), $v$ sends $\frac{1}{2}$ to each incident $4$-face and each neighbor of $v$  sends $\frac{3}{2}$ to $v$. Thus,  $\mu^*(f)=2\times2-6-\frac{1}{2}\times2+\frac{3}{2}\times2=0$. If $v$ is semi-bad, then $v$ is incident with one $4$-face. By (R1) and (R2), $v$ sends $\frac{1}{2}$ to each incident $4$-face and each neighbor of $v$  sends $\frac{5}{4}$ to $v$. Thus,  $\mu^*(f)=2\times2-6-\frac{1}{2}+\frac{5}{4}\times2=0$. If $v$ is non-bad, then $v$ is incident with two $5^+$-faces. By   (R2),   each neighbor of $v$  sends $1$ to $v$. Thus,  $\mu^*(v)=2\times2-6+1\times2=0$.
  \item Let $v$ be a $4$-vertex. By (R1), $v$  sends at most $\frac{1}{2}$ to each incident face.  By Lemma \ref{tool}(3), $v$ is adjacent to at most one $2$-vertex or convenient vertex. If $v$ is adjacent to  a $2$-vertex, then $v$ is convenient and  adjacent to three non-convenient  vertices.  By (R2), $v$ sends at most $\frac{3}{2}$ to $2$-neighbor. By (R3), each non-convenient neighbor of $v$ sends $\frac{1}{2}$ to $v$. Thus, $\mu^*(v)=2\times4-6-\frac{1}{2}\times4-\frac{3}{2}+\frac{1}{2}\times3=0$. If $v$ is not adjacent to  any $2$-vertex, then $v$ is non-convenient  and  adjacent to at most one  convenient  vertex. If $v$ is adjacent to   one  convenient  vertex, then $v$ is adjacent to three non-convenient $6^+$-vertices by Lemma \ref{n-c-4-v}.  By (R3) and (R4), each   $v$ sends $\frac{1}{2}$ to convenient neighbor and each non-convenient $6^+$-neighbors sends $\frac{1}{4}$ to $v$. Thus, $\mu^*(v)=2\times4-6-\frac{1}{2}\times4-\frac{1}{2}+\frac{1}{4}\times3>0$. If $v$ is not adjacent to   any  convenient  vertex,  then $\mu^*(v)=2\times4-6-\frac{1}{2}\times4=0$.
  \item Let $v$ be a $5$-vertex. Then $v$ is convenient.    By (R1), $v$  sends   $\frac{1}{2}$ to each incident $4$-face.  If $v$ is not adjacent to any $2$-vertex, then $\mu^*(v)\geq2\times5-6-\frac{1}{2}\times5>0$. By Lemma \ref{tool}(3), $v$ is adjacent to at most three $2$-vertices. If $v$ is adjacent to at most two $2$-vertices, then $v$ is adjacent to three non-convenient vertices by Lemma \ref{tool}(2).  By (R2), $v$ sends at most $\frac{3}{2}$ to $2$-neighbor. By (R3), each non-convenient neighbor of $v$ sends $\frac{1}{2}$ to $v$. Thus, $\mu^*(v)\geq2\times5-6-\frac{1}{2}\times5-\frac{3}{2}\times2+\frac{1}{2}\times3=0$.
      If $v$ is adjacent to three $2$-vertices, then  $v_1,v_2,v_3$   or $v_1,v_2,v_4$ are $2$-vertices by symmetry. In the former case,  $f_1$ and $f_2$ are $6^+$-faces by Lemma \ref{4-f} and $v_1$ and $v_3$ are not bad and $v_2$ is non-bad. Then $v$ is incident with at most three $4$-faces. By (R2), $v$ sends at most $\frac{5}{4}$ to each of $v_1$ and $v_3$, $1$ to $v_2$. By Lemma \ref{tool}(2), $v_4$ and $v_5$ are non-convenient. By (R3), each of  $v_4$ and $v_5$ sends $\frac{1}{2}$ to $v$. Thus, $\mu^*(v)\geq2\times5-6-\frac{1}{2}\times3-\frac{5}{4}\times2-1+\frac{1}{2}\times2=0$.
      In the latter case,   $f_1$ is a $6^+$-face by Lemma \ref{4-f} and $v_3$ and $v_5$ are non-convenient vertices. If each of $f_2,f_3,f_4,f_6$ is a $4$-face, then $v$ is poor to $v_3$ and $v_5$, each of $v_1$ and $v_2$ is semi-bad  and $v_4$ is bad. By (R2), $v$ sends $\frac{5}{4}$ to each of $v_1$ and $v_2$, $\frac{3}{2}$ to $v_4$. By (R3), each of $v_3$ and $v_5$ sends $1$ to $v$. Thus, $\mu^*(v)\geq2\times5-6-\frac{1}{2}\times4-
      \frac{5}{4}\times2-\frac{3}{2}+1\times2=0$.
      If three of $f_2,f_3,f_4,f_6$ are $4$-faces,  then at most one of $v_3$ and $v_5$ is poor to $v$.  Thus, $\mu^*(v)\geq2\times5-6-\frac{1}{2}\times3-
      \frac{5}{4}\times2-\frac{3}{2}+1+\frac{1}{2}=0$.
      If at most two of $f_2,f_3,f_4,f_6$ are   $4$-faces,  then  $\mu^*(v)\geq2\times5-6-\frac{1}{2}\times2-
      \frac{5}{4}\times2-\frac{3}{2}+\frac{1}{2}\times2=0$.
  \item Let $v$ be a $6^+$-vertex, we use $k_j$-vertex to denote $v$, where $k\ge 6$, $0\le j\le k$. 

      If $j=k$, then each face incident with $v$ is a $6^+$-face by Lemma \ref{4-f}. Then each $2$-neighbor of $v$ is non-bad. By (R2), $v$ sends $1$ to each $2$-neighbor. Thus,  $\mu^*(v)\geq 2k-6-k\geq0$, equal holds when $k=6$. Note that $6_6$-vertex does not exist by Lemma~\ref{tool}(3), $\mu^*(v)>0$.
     
            If $0< j< k$, 
  then $v$ is convenient and adjacent to $k-j$ non-convenient vertices when $j\notin \{0, k\}$. By (R3), each non-convenient neighbor of $v$ sends $\frac{1}{2}$ to $v$.
If $1\leq j\leq \frac{k}{2}$, then each $2$-neighbor of $v$ can be bad and each face incident with $v$ can be $4$- or $5$-face. By (R1) and (R2),  $v$ sends at most $\frac{3}{2}$ to each $2$-neighbor and at most $\frac{1}{2}$ to each incident face. Thus,  $\mu^*(v)\geq 2k-6+\frac{k-j}{2}-\frac{3j}{2}-\frac{k}{2}\geq\frac{5k}{4}-6>0$.
If $\frac{k}{2}< j<k$, then $v$ is adjacent to at least $j-\frac{k}{2}$  $2$-vertices $v_i$ such that $v_{i-1}$ and $v_{i+1}$ $2$-vertices (We get this by pigeon principle: arrange $2$-vertex/$4^+$-vertex in cross, then change $4^+$-vertex to $2$-vertex). Thus, $v$ is incident with at least $2(j-\frac{k}{2})$ $6^+$-faces by Lemma \ref{4-f}. Then $v$ is adjacent to at least $j-\frac{k}{2}$  non-bad $2$-neighbors, at most $2(j-\frac{k}{2})$ semi-bad $2$-neighbors. Then $v$ is incident with at most $k-2(j-\frac{k}{2})=2k-2j$ $5^-$-faces, at most $j-(j-\frac{k}{2})-2(j-\frac{k}{2})=\frac{3k}{2}-2j$ bad $2$-neighbors. By (R1), (R2) and (R3), $\mu^*(v)\geq 2k-6+\frac{k-j}{2}-(j-\frac{k}{2})-\frac{5}{4}\times2(j-\frac{k}{2})
      -\frac{3}{2}\times (\frac{3k}{2}-2j)-\frac{1}{2}\times (2k-2j) \geq k-6>0$.

      If $j=0$, then $v$ is not adjacent to any $2$-neighbors. $v$ is convenient if $v$ is an odd vertex and non-convenient otherwise. If $v$ is an odd vertex, then $v$ sends at most $\frac{1}{2}$ to each incident face by (R1). Thus, $\mu^*(v)\geq 2d(v)-6-\frac{d(v)}{2}>0$. If $v$ is an even vertex, then $v$ sends at most $1$ to each convenient neighbor by (R3) and (R4). Thus, $\mu^*(v)\geq 2d(v)-6-\frac{d(v)}{2}-d(v)\geq0$. Then we analyze the cases $k=6, 8, 10$, separately.

   If  $k=10$, then if $v$ is adjacent to  ten  convenient neighbors, then $v$ is not adjacent to any $5_3$-vertex by Lemma \ref{6-v}. Then $v$ sends at most $\frac{1}{2}$ to each  convenient neighbor. Thus, $\mu^*(v)\geq 2\times8-6-\frac{d(v)}{2}-\frac{d(v)}{2}>0$. If $v$ is adjacent to nine  convenient neighbors, then $v$ is adjacent to one non-convenient neighbor.   We assume that  $v_i$   is non-convenient.  Then $v_{i-1}$ and $v_{i+1}$ are not poor to $v$. Otherwise $f_{i-1}$ or  $f_i$ is a $4_1$-face. Then $v_i$ is convenient, a contradiction.  Thus, $v$ sends at most $\frac{1}{2}$ to $v_{i-1}$ and $v_{i+1}$ by (R3). By (R4), $v$ sends at most $\frac{1}{4}$ to $v_i$. Thus, $\mu^*(v)\geq 2\times10-6-\frac{d(v)}{2}-7\times1-\frac{1}{2}\times2-\frac{1}{4}>0$.  If $v$ is   adjacent to  at most eight convenient neighbors,  then $v$ sends at most $1$ to convenient neighbors, at most $\frac{1}{4}$ to non-convenient neighbors by (R3) and (R4). Thus,  $\mu^*(v)\geq 2\times10-6-\frac{10}{2}-8\times1-\frac{1}{4}\times2>0$.

       If $k=8$, then if $v$ is  adjacent to at least seven  convenient neighbors, then $v$ is not adjacent to any $5_3$-vertex by Lemma \ref{6-v}. Then $v$ sends at most $\frac{1}{2}$ to each  convenient neighbor. Thus, $\mu^*(v)\geq 2\times8-6-\frac{8}{2}-\frac{8}{2}>0$. If $v$ is adjacent to   six convenient neighbors,  then $v$ is adjacent to two non-convenient neighbors.  We assume that  $v_i$ and  $v_j$ are non-convenient where $i<j$.  Then $v_{i-1}$ is not poor to $v$. Otherwise $f_{i-1}$ is a $4_1$-face. Then $v_i$ is convenient, a contradiction. By symmetry, $v_{j+1}$ is not poor to $v$. Thus, $v$ sends at most $\frac{1}{2}$ to $v_{i-1}$ and $v_{j+1}$ by (R3). By (R4), $v$ sends at most $\frac{1}{4}$ to $v_i$ and $v_j$. Thus, $\mu^*(v)\geq 2\times8-6-\frac{8}{2}-4\times1-\frac{1}{2}\times2-\frac{1}{4}\times2>0$. If $v$ is   adjacent to  at most five convenient neighbors,  then $v$ sends at most $1$ to convenient neighbors, at most $\frac{1}{4}$ to non-convenient neighbors by (R3) and (R4). Thus,  $\mu^*(v)\geq 2\times8-6-\frac{8}{2}-5\times1-\frac{1}{4}\times3>0$.

       If $k=6$, then by Lemma \ref{tool}(3), $v$ is adjacent to at most five convenient neighbors. If $v$ is  adjacent to   four or five convenient neighbors, then $v$ is not adjacent to any $5_3$-vertex by Lemma \ref{6-v}. By (R3) and (R4), $v$ sends at most $\frac{1}{2}$ to each  convenient neighbor, at most $\frac{1}{4}$ to each non-convenient neighbor. Thus, $\mu^*(v)\geq 2\times6-6-\frac{6}{2}-\{\frac{1}{2}\times5+\frac{1}{4}, \frac{1}{2}\times4+\frac{1}{4}\times2\} >0$. If $v$ is adjacent to   three  convenient neighbors,  then $v$ is adjacent to three non-convenient neighbors.  We assume that  $v_i$ and  $v_j$ are non-convenient where $i<j$.  Then $v_{i-1}$ is not poor to $v$. Otherwise $f_{i-1}$ is a $4_1$-face. Then $v_i$ is convenient, a contradiction. By Symmetry, $v_{j+1}$ is not poor to $v$. Thus, $v$ sends at most $\frac{1}{2}$ to $v_{i-1}$ and $v_{j+1}$ by (R3). By (R4), $v$ sends at most $\frac{1}{4}$ to $v_i$ and $v_j$. Then $v$ is adjacent to at most one poor $5_3$-vertex. Thus, $\mu^*(v)\geq 2\times6-6-\frac{6}{2}-1-\frac{1}{2}\times2-\frac{1}{4}\times3>0$. If $v$ is   adjacent to  at most two convenient neighbors,  then $v$   is not adjacent to any poor $5_3$-vertex.  Thus,  $\mu^*(v)\geq 2\times8-6-\frac{6}{2}- \frac{1}{2}\times2-\frac{1}{4}\times4>0$.
  \end{enumerate}

  \begin{lem}\label{12-vert}
 If $G$ has a non-convenient $12$-vertex $v$ and $v$ is   adjacent to twelve poor $5_3$-vertices, then $\sum_{v\in V }\mu^*(v)+\sum_{f\in F }\mu^*(f)>0$.
\end{lem}
\begin{proof}
  Let $v_1,v_2,\ldots,v_{12}$ be the neighbors of $v$, $f_i$ be the face incident with $v_ivv_{i+1}$, $u_i$ be the $2$-vertex of $f_i$, $f_i'\neq f_i$ be the face incident with $v_iu_iv_{i+1}$  for $i=\{1,2,\ldots,12\}$ and $i+1=1$ if $i=12$. Then each   of $v_i$ is a $5_3$-vertex for $i=\{1,2,\ldots,12\}$. By Lemma \ref{5-f},  each $f_i'$ is a $5^+$-face.  Since $v_i$ is a $5_3$-vertex, then at most one of $f_i'$ and $f_{i-1}'$ is a $6_3$-face. Then $v_i$ sends no charge to at least one of $f_i'$ and $f_{i-1}'$. Then each $2$-neighbors of $v_i$ is semi-bad and $v_i$ is incident with at most three $4$-faces.   Thus, $\mu^*(v_i)\geq 2\times5-\frac{5}{4}\times3-\frac{1}{2}\times3+1+\frac{1}{2}>0$. Thus, if $G$ has a non-convenient $12$-vertex $v$, then $\sum_{v\in V }\mu^*(v)+\sum_{f\in F }\mu^*(f)>0$.
\end{proof}

By Lemma~\ref{tool}, $G$ must either be formed by $k_k$-vertices for $k\geq7$, or have non-convenient vertices. In the former case, the final charge of $G$ is always positive. In the latter case, by the above checking procedure and Lemma \ref{12-vert}, $v$ is  non-convenient and $\mu^*(v)=0$ if and only if   $d(v)=4$ and $v$ is not adjacent to any convenient vertices. Thus, $G$ has only non-convenient $4$-vertices.  If $G$ has a $5$-face $f$, then $\mu^*(f)=5-6+\frac{1}{2}\times5>0$.  If $G$ has a $6^+$-face $f=[v_1v_2\ldots v_6]$, then  $\mu^*(v_i)=2\times4-6-\frac{1}{2}\times3>0$ for each $i=1,2\ldots,6$. By Lemma \ref{4_0-face}, $G$ has no $(4_0,4_0,4_0,4_0)$-face. Thus, $\sum_{v\in V }\mu^*(v)+\sum_{f\in F }\mu^*(f)>0$, a contradiction.


\small

\begin{thebibliography}{99}

\bibitem{smorodinsky2013conflict}
Smorodinsky, Shakhar. "Conflict-free coloring and its applications." Geometry—Intuitive, Discrete, and Convex. Springer, Berlin, Heidelberg, 2013. 331-389.
\bibitem{petruvsevski2021colorings}
Petru$\breve{s}$evski, Mirko, and Riste $\breve{S}$krekovski. "Colorings with neighborhood parity condition." arXiv preprint arXiv:2112.13710 (2021).
\bibitem{petr2022odd}
Petr, Jan, and Julien Portier. "Odd chromatic number of planar graphs is at most 8." arXiv preprint arXiv:2201.12381 (2022).
\bibitem{fabrici2022proper}
Fabrici, Igor, et al. "Proper conflict-free and unique-maximum colorings of planar graphs with respect to neighborhoods." arXiv preprint arXiv:2202.02570 (2022).
\bibitem{cranston2022odd}
Cranston, Daniel W. "Odd Colorings of Sparse Graphs." arXiv preprint arXiv:2201.01455 (2022).
\bibitem{cho2022odd}
Cho, Eun-Kyung, et al. "Odd coloring of sparse graphs and planar graphs." arXiv preprint arXiv:2202.11267 (2022).
\bibitem{kauffman2009seven}
Kauffman, Louis H. "Seven Knots and Knots in the Seven-Color Map." Homage to a Pied Puzzler (2009): 75.
\end{thebibliography}
  \end{document}